\documentclass[12pt]{article}
\usepackage[a4paper, total={6in, 8in}, margin={1in}]{geometry}
\usepackage[utf8]{inputenc}

\usepackage{amsmath}
\usepackage{amsthm}
\usepackage{amssymb}
\usepackage{amsfonts}
\usepackage[hidelinks]{hyperref}
\usepackage{scrextend}
\usepackage{epsfig}
\usepackage{forest}

\theoremstyle{plain}
\newtheorem{theorem}{Theorem}

\newtheorem{lemma}[theorem]{Lemma}

\theoremstyle{definition}

\theoremstyle{remark}
\newtheorem{remark}[theorem]{Remark}

\title{Binary Trees and Sum of Two Squares}
\author{Hongshen Chua}
\date{}

\begin{document}
	\maketitle

\section{Introduction}

When it comes to binary trees of rational numbers, there are two that stand out in particular, namely the Stern--Brocot tree (Figure \ref{Fg: Stern Brocot}) and the Calkin--Wilf tree (Figure \ref{Fg: Calkin Wilf}).

Both trees start with the number $ 1 $ at the $ 0 $-th level. The $ (i + 1) $-th level of the Stern-Brocot tree is generated from two adjacent terms, $ a_1/b_1 $ and $ a_2/b_2 $, on the $ i $-th level by computing their mediant, $ (a_1 + a_2)/(b_1 + b_2) $. In contrast, to construct the Calkin--Wilf tree, we take a vertex $ a/b $ at the $ i $-th level, then set its left child as $ a/(a + b) $ and its right child as $ (a + b)/b $.
\begin{figure}[h]
\centering
\scalebox{1}{
\begin{forest}
	[$ \frac{1}{1} $
		[$ \frac{1}{2} $
			[$ \frac{1}{3} $
				[$ \frac{1}{4} $]
				[$ \frac{2}{5} $]
			]
			[$ \frac{2}{3} $
				[$ \frac{3}{5} $]
				[$ \frac{3}{4} $]
			]
		]
		[$ \frac{2}{1} $
			[$ \frac{3}{2} $
				[$ \frac{4}{3} $]
				[$ \frac{5}{3} $]
			]
			[$ \frac{3}{1} $
				[$ \frac{5}{2} $]
				[$ \frac{4}{1} $]
			]
		]
	]
\end{forest}
}
\caption{Stern--Brocot tree} \label{Fg: Stern Brocot}
\end{figure}
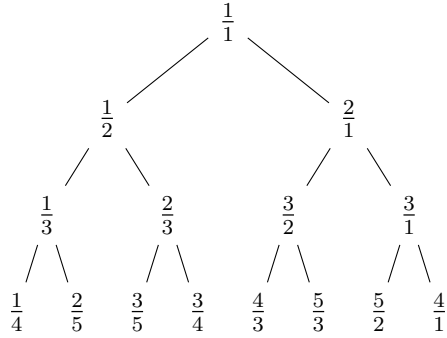

\begin{figure}[h]
\centering
\scalebox{1}{
\begin{forest}
	[$ \frac{1}{1} $
		[$ \frac{1}{2} $
			[$ \frac{1}{3} $
				[$ \frac{1}{4} $]
				[$ \frac{4}{3} $]
			]
			[$ \frac{3}{2} $
				[$ \frac{3}{5} $]
				[$ \frac{5}{2} $]
			]
		]
		[$ \frac{2}{1} $
			[$ \frac{2}{3} $
				[$ \frac{2}{5} $]
				[$ \frac{5}{3} $]
		]
			[$ \frac{3}{1} $
				[$ \frac{3}{4} $]
				[$ \frac{4}{1} $]
			]
		]
	]
\end{forest}
}
\caption{Calkin--Wilf tree} \label{Fg: Calkin Wilf}
\end{figure}
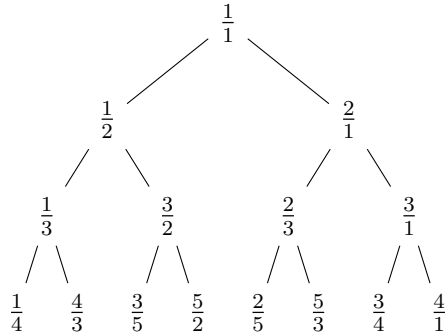

Both trees appear quite similar, so it is perhaps unsurprising that they are closely related. In fact, Backhouse and Ferreira \cite{Backhouse_Ferreira} described a remarkable matrix-based binary tree that naturally generates both the Stern--Brocot and Calkin--Wilf trees. We will explore this construction in more detail in Section \ref{Ch: Binary Tree}. 

The matrix binary tree also integrates seamlessly with the theory of continued fractions. The convergents of a continued fraction can be derived from a matrix similar to the one we use to generate our tree. Consequently, these convergents can be represented as paths through the matrix binary tree. This will be the focus of Section \ref{Ch: Continued Fraction}. 

Finally, Brillhart wrote a brief note \cite{Brillhart} presenting a proof of Fermat's famous theorem on the sum of two squares. The main idea of his method relies on continued fractions. Our path-based approach offers a new interpretation of his proof, which we will present in Section \ref{Ch: Sum of Two Squares}. 

\section{Binary Trees} \label{Ch: Binary Tree}

First, we construct a matrix binary tree that underlies both the Stern--Brocot and Calkin--Wilf trees. Starting with the $ 2 \times 2 $ identity matrix, a right ($ R $) and a left ($ L $) move correspond to post-multiplication by the matrices $ \begin{pmatrix} 1 & 1 \\ 0 & 1 \end{pmatrix} $ and $ \begin{pmatrix} 1 & 0 \\ 1 & 1 \end{pmatrix} $, respectively. Alternatively, an $ R $-move can be thought of as adding the left column to the \textit{right} column, while an $ L $-move involves adding the right column to the \textit{left} column. By repeatedly applying these moves, we can construct the matrix binary tree shown in Figure \ref{Fg: Matrix binary}.
\begin{figure}[h]
\centering
\scalebox{0.8}{
\begin{forest}
	[$ \begin{pmatrix} 1 & 0 \\ 0 & 1 \end{pmatrix} $
		[$ \begin{pmatrix} 1 & 0 \\ 1 & 1 \end{pmatrix} $
			[$ \begin{pmatrix} 1 & 0 \\ 2 & 1 \end{pmatrix} $
				[$ \begin{pmatrix} 1 & 0 \\ 3 & 1 \end{pmatrix} $]
				[$ \begin{pmatrix} 1 & 1 \\ 2 & 3 \end{pmatrix} $]
			]
			[$ \begin{pmatrix} 1 & 1 \\ 1 & 2 \end{pmatrix} $
				[$ \begin{pmatrix} 2 & 1 \\ 3 & 2 \end{pmatrix} $]
				[$ \begin{pmatrix} 1 & 2 \\ 1 & 3 \end{pmatrix} $]
			]
		]
		[$ \begin{pmatrix} 1 & 1 \\ 0 & 1 \end{pmatrix} $
			[$ \begin{pmatrix} 2 & 1 \\ 1 & 1 \end{pmatrix} $
				[$ \begin{pmatrix} 3 & 1 \\ 2 & 1 \end{pmatrix} $]
				[$ \begin{pmatrix} 2 & 3 \\ 1 & 2 \end{pmatrix} $]
			]
			[$ \begin{pmatrix} 1 & 2 \\ 0 & 1 \end{pmatrix} $
				[$ \begin{pmatrix} 3 & 2 \\ 1 & 1 \end{pmatrix} $]
				[$ \begin{pmatrix} 1 & 3 \\ 0 & 1 \end{pmatrix} $]
			]
		]
	]
\end{forest}
}
\caption{Matrix binary tree} \label{Fg: Matrix binary}
\end{figure}
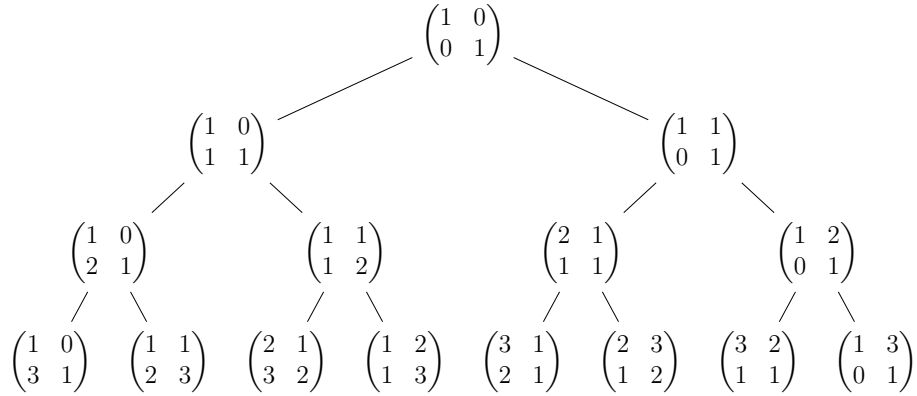

From the matrix binary tree, we can easily recover the Stern--Brocot and Calkin--Wilf trees. By post-multiplying each matrix by $ \begin{pmatrix} 1 \\ 1 \end{pmatrix} $, which is equivalent to adding the columns together, we obtain the Stern--Brocot tree (Figure \ref{Fg: Stern Brocot}). On the other hand, if we pre-multiply each matrix by $ \begin{pmatrix} 1 & 1 \end{pmatrix} $, which corresponds to adding the rows together, we get the Calkin--Wilf tree (Figure \ref{Fg: Calkin Wilf}).

\section{Continued Fraction} \label{Ch: Continued Fraction}

Let $ [q_0; q_1, q_2, \ldots] $ represent the continued fraction
\begin{align*}
	q_0 + \cfrac{1}{q_1 + \cfrac{1}{q_2 + \ddots}}.
\end{align*}
The \textit{$ n $-th convergent} $ A_n/B_n $ of a continued fraction is the fraction obtained by truncating the continued fraction at $ q_n $, i.e., $ A_n/B_n = [q_0; q_1, q_2, \ldots, q_n] $. These convergents are related by the following linear recurrences:
\begin{align*}
	A_n = q_n A_{n - 1} + A_{n - 2} \quad \text{ and } \quad B_n = q_n B_{n - 1} + B_{n - 2},
\end{align*}
with the initial values $ A_{-1} = 1 $, $ A_0 = q_0 $, $ B_{-1} = 0 $, and $ B_0 = 1 $.

The matrix binary tree provides a powerful representation of the convergents. Given a continued fraction $ [q_0; q_1, q_2, \ldots] $, we define a path $ W $ through the tree by starting with $ q_0 $ $ R $-moves, followed by $ q_1 $ $ L $-moves, and so on, alternating the direction after each term. By an abuse of notation, we write $ R = \begin{pmatrix} 1 & 1 \\ 0 & 1 \end{pmatrix} $, $ L = \begin{pmatrix} 1 & 0 \\ 1 & 1 \end{pmatrix} $, and the matrix corresponding to the path $ W $ as $ R^{q_0} L^{q_1} \ldots $. This path produces the convergents, as demonstrated in the next theorem.
\begin{theorem} \label{Th: Convergent Matrix} 
Let $ A_n/B_n $ be the $ n $-th convergent of a continued fraction $ [q_0; q_1, q_2, \ldots] $. If a path $ W_n = R^{q_0} L^{q_1} \ldots $ ends with an $ R $-move, then 
\begin{align*}
	W_n = \begin{pmatrix} A_{n - 1} & A_n \\ B_{n - 1} & B_n \end{pmatrix}.
\end{align*}
Otherwise, if $ W_n $ ends with an $ L $-move, then 
\begin{align*}
	W_n = \begin{pmatrix} A_n & A_{n - 1} \\ B_n & B_{n - 1} \end{pmatrix}.
\end{align*}
\end{theorem}

\begin{proof} A quick calculation shows that consecutive $ R $ or $ L $ moves have a simple matrix representation of
\begin{align*}
	R^q = \begin{pmatrix} 1 & q \\ 0 & 1 \end{pmatrix} \quad \text{ and } \quad L^q = \begin{pmatrix} 1 & 0 \\ q & 1 \end{pmatrix}.
\end{align*}

We proceed by induction. For the base case of $ n = 0 $, we have
\begin{align*}
	W_0 = R^{q_0} = \begin{pmatrix} 1 & q_0 \\ 0 & 1 \end{pmatrix} = \begin{pmatrix} A_{-1} & A_0 \\ B_{-1} & B_0 \end{pmatrix}.
\end{align*}
Next, it is simple to verify that
\begin{align*}
	A_1 = q_1 A_0 + A_{-1} = q_0 q_1 + 1 \quad \text{ and } \quad B_1 = q_1 B_0 + B_{-1} = 1.
\end{align*}
Hence, for the case $ n = 1 $, we compute
\begin{align*}
	R^{q_0} L^{q_1} = \begin{pmatrix} 1 & q_0 \\ 0 & 1 \end{pmatrix} \begin{pmatrix} 1 & 0 \\ q_1 & 1 \end{pmatrix} = \begin{pmatrix} q_0 q_1 + 1 & q_0 \\ q_1 & 1 \end{pmatrix} = \begin{pmatrix} A_1 & A_0 \\ B_1 & B_0 \end{pmatrix}.
\end{align*}
This completes the base cases for $ n = 0 $ and $ n = 1 $.

Now, suppose the result holds for all $ n \leq N $ for some $ N \geq 1 $. We prove it for $ n = N + 1 $. If $ W_{N + 1} $ ends with an $ R $-move, then $ W_N $ must have ended with an $ L $-move. By the inductive hypothesis, we get
\begin{align*}
	W_{N + 1} = \begin{pmatrix} A_N & A_{N - 1} \\ B_N & B_{N - 1} \end{pmatrix} \begin{pmatrix} 1 & q_{N + 1} \\ 0 & 1 \end{pmatrix} = \begin{pmatrix} A_N & q_{N + 1} A_N + A_{N - 1} \\ B_N & q_{N + 1} B_N + B_{N - 1} \end{pmatrix}.
\end{align*}
Using the recurrence relation for convergents, the last matrix simplifies to
\begin{align*}
	W_{N + 1} &= \begin{pmatrix} A_N & A_{N + 1} \\ B_N & B_{N + 1} \end{pmatrix}.
\end{align*}
On the other hand, if $ W_{N + 1} $ ends with an $ L $-move, then $ W_N $ must have ended with an $ R $-move, so 
\begin{align*}
	W_{N + 1} = \begin{pmatrix} A_{N - 1} & A_N \\ B_{N - 1} & B_N \end{pmatrix} \begin{pmatrix} 1 & 0 \\ q_{N + 1} & 1 \end{pmatrix} = \begin{pmatrix} q_{N + 1} A_N + A_{N - 1} & A_N \\ q_{N + 1} B_N + B_{N - 1} & B_N \end{pmatrix}.
\end{align*}
Again, the recurrence relation for convergents yields
\begin{align*}
	W_{N + 1} = \begin{pmatrix} A_{N + 1} & A_N \\ B_{N + 1} & B_N \end{pmatrix}.
\end{align*}
This completes the induction step, and hence the proof. 
\end{proof}

\section{Sum of Two Squares} \label{Ch: Sum of Two Squares}

In a letter to Mersenne, Fermat stated that any prime $ p $ of the form $ 4k + 1 $ can be written as a sum of two squares \cite{Dickson}. However, like many of his conjectures, Fermat did not provide a proof for this result. The first complete proof was given by Euler, using the method of infinite descent, a technique, rather fittingly, invented by Fermat. 

Our approach is inspired by Brillhart's excellent note  \cite{Brillhart} on this subject. First, we find a solution $ x_0 $ to the congruence $ x^2 \equiv -1 \pmod{p} $, where $ 0 < x_0 < p/2 $. Such a solution always exists as $ -1 $ is a quadratic residue modulo $ p $. Specifically, $ (-1 \mid p) = (-1)^{(p - 1)/2} = 1 $, where $ (\cdot \mid p) $ denotes the Legendre symbol modulo $ p $.

Once we have chosen our $ x_0 $, a result from Perron states that the continued fraction of $ p/x_0 $ has a particularly neat structure. 
\begin{lemma}[\cite{Perron}] \label{Th: Palindromic continued fraction}
The continued fraction of $ p/x_0 $ has an even number of terms and is palindromic, i.e., 
\begin{align*}
	p/x_0 = [q_0, q_1, \ldots, q_n, q_n, \ldots, q_1, q_0].
\end{align*}
\end{lemma}

So far, our proof follows Brillhart's approach, but this is the point where we diverge. Rather than working with continuants, consider the path
\begin{align*}
	W = R^{q_0} L^{q_1} \ldots S^{q_n} S^{q_n} \ldots R^{q_1} L^{q_0},
\end{align*} 
where $ S $ is either $ R $ or $ L $, depending on whether $ n $ is odd or even. Since the path $ W $ ends with an $ L $-move, Theorem \ref{Th: Convergent Matrix} tells us that 
\begin{align*}
	W = \begin{pmatrix} p & * \\ x_0 & * \end{pmatrix}.
\end{align*}

On the other hand, let $ M = R^{q_0 - 1} L^{q_1} \ldots S^{q_n} = \begin{pmatrix} a & b \\ c & d \end{pmatrix} $, where $ a $, $ b $, $ c $, and $ d $ are integers. Since $ R = L^T $, we have $ W = RMM^TL $. The middle term $ MM^T $ evaluates to 
\begin{align*}
	MM^T &= \begin{pmatrix} a & b \\ c & d \end{pmatrix} \begin{pmatrix} a & c \\ b & d \end{pmatrix} = \begin{pmatrix} a^2 + b^2 & ac + bd \\ ac + bd & c^2 + d^2 \end{pmatrix}.
\end{align*}
Therefore, we have
\begin{align*}
	RMM^TL = \begin{pmatrix} 1 & 1 \\ 0 & 1 \end{pmatrix} \begin{pmatrix} a^2 + b^2 & ac + bd \\ ac + bd & c^2 + d^2 \end{pmatrix} \begin{pmatrix} 1 & 0 \\ 1 & 1 \end{pmatrix},
\end{align*}
which simplifies to 
\begin{align*}
	W = \begin{pmatrix} (a + c)^2 + (b + d)^2 & (a + c)c + (b + d)d \\ (a + c)c + (b + d)d & c^2 + d^2 \end{pmatrix}.
\end{align*}
Comparing both expressions for $ W $ yields
\begin{align*}
	p = (a + c)^2 + (b + d)^2.
\end{align*}
Thus, we conclude the following theorem.
\begin{theorem}
	If a prime $ p $ has the form $ 4k + 1 $, then $ p $ is a sum of two squares. 
\end{theorem}

\begin{remark}
	We also have the relation $ x_0 = (a + c)c + (b + d)d $.  
\end{remark}

An advantage of this method is that we can explicitly identify the two squares that make up the sum. The numbers $ (a + c) $ and $ (b + d) $ correspond to the path $ M^T $ applied to the matrix $ \begin{pmatrix} 1 \\ 1 \end{pmatrix} $. Hence, we can state the following theorem.
\begin{theorem}
	If the fraction $ x/y $ corresponds to the path $ M^T $ in the Stern--Brocot tree, then $ p = x^2 + y^2 $. 
\end{theorem}

\begin{remark}
	The same holds for the Calkin--Wilf tree, but with the path $ M $ instead.
\end{remark}

\vfill

\end{document}